\documentclass[a4paper]{amsart}

\usepackage{latexsym, amssymb,amsmath, amsthm,amsfonts}

\newcommand{\kk}{\Bbbk}

\newcommand{\kvg}{{\kk[V]^{G}}}

\newcommand{\NGV}{{\mathcal{N}_{G,V}}}

\def\SL{\operatorname{SL}}

\def\SL2{\operatorname{SL}_{2}(K)}

\def\GL2{\operatorname{GL}_{2}(K)}

\def\INVSL2{$K[V]^{operatorname{SL}_{2}(K)}$}
\def\INVSO2{$K[V]^{operatorname{SO}_{2}(K)}$}
\def\INVGL2{$K[V]^{operatorname{GL}_{2}(K)}$}

\def\GL{\operatorname{GL}}
\def\SL{\operatorname{SL}}

\def\chr{\operatorname{char}}

\def\Z{\mathbb{Z}}
\def\N{\mathbb{N}}

\newtheorem{Lemma}{Lemma}[section]
\newtheorem{Theorem}[Lemma]{Theorem}
\newtheorem{Corollary}[Lemma]{Corollary}

\newtheorem*{Corollary of Conjecture}{Corollary of Conjecture}

\theoremstyle{definition}

\theoremstyle{remark}
  \newtheorem{rem}[Lemma]{Remark}
\newtheorem{eg}[Lemma]{Example}

\newtheoremstyle{Acknowledgments}
  {}
    {}
     {}
     {}
    {\bfseries}
    {}
     {.5em}
     {\thmname{#1}\thmnumber{ }\thmnote{ (#3)}}
\theoremstyle{Acknowledgments}
\newtheorem{ack}{Acknowledgments.}

\title[Separating fixed points from zero]
{On separating a fixed point from zero by invariants}

\author{Jonathan Elmer}
\address{University of Aberdeen\\
King's College, Aberdeen\\
AB24 3UE}
\email{j.elmer@abdn.ac.uk}

\author{Martin Kohls}
\address{Technische Universit\"at M\"unchen \\
 Zentrum Mathematik-M11\\
Boltzmannstrasse 3\\
 85748 Garching, Germany}
\email{kohls@ma.tum.de}

\date{\today}

\begin{document}
\maketitle

\begin{abstract}
Assume a fixed point $v \in V^{G}$ can be separated from zero by a homogeneous invariant $f\in\kk[V]^{G}$ of degree $p^{r}d$ where $p>0$ is the characteristic of the
ground field $\kk$ and $p,d$ are coprime. We show that then $v$ can also be
separated from zero by an invariant of degree $p^{r}$, which we obtain
explicitly from $f$. It follows that the minimal degree of a homogeneous
invariant separating $v$ from zero is a $p$-power.
\end{abstract}

\section{Introduction}\label{SecIntro}

Let $G$ be a linear algebraic group over an infinite field $\kk$ of any characteristic and let $X$ be an algebraic variety over $\kk$ on which $G$ acts.  Then $G$
acts naturally on the ring of functions $\kk[X]$ by $g(f):=f\circ g^{-1}$
for $f\in\kk[X]$ and $g\in G$. The ring of fixed points of this action is
denoted by $\kk[X]^G$ and we call this the ring of invariants. If $G$ acts
linearly and rationally on a finite dimensional $\kk$-vector space $V$ then we call $V$ a $G$-module, and $\kk[V]$ is the set of polynomial functions $V \rightarrow \kk$. In that case we have a natural grading $\kk[V]=\oplus_{d=0}^{\infty}\kk[V]_{d}$
by total degree which is preserved by the action of $G$, and we have
$\kk[V]=S(V^*)$, the symmetric algebra of the dual of $V$.
Determining whether the ring of invariants $\kk[X]^G$ is finitely generated is
one of the oldest and most difficult problems in invariant theory. Hilbert was
able to prove finite generation in the case where $G=\SL_n$ or $\GL_n$ and
$\kk$ a field of characteristic zero. Hilbert's argument can be extended to
any group with the following property: for every $G$-module $V$, and every nonzero fixed point $v \in V^G$ there exists an invariant linear function $f \in (V^*)^G$ such that $f(v) \neq 0$. Such groups are called \emph{linearly reductive}. Linear reductivity of $G$ is equivalent to the condition that all representations of $G$ over $\kk$ are completely reducible.
Nagata made a major breakthrough by considering a more general class of
groups. We say that $G$ is \emph{geometrically reductive} if the following
property holds: for every $G$-module $V$ and every nonzero fixed point $v \in V^G$
there exists a homogeneous invariant function $f \in \kk[V]^G$ of positive degree such
that $f(v) \neq 0$. Nagata~\cite{NagataAffine} was able to prove that if $G$ is geometrically reductive then $\kk[X]^G$ is finitely generated for all $X$. Nagata and Miyata \cite{NagataMiyata} subsequently showed that a geometrically reductive group must be reductive, a purely group-theoretic condition on $G$. It was conjectured by Mumford \cite{Mumford} that all reductive groups are geometrically reductive, a fact finally proved by Haboush \cite{Haboush} several years later.
Now let $G$ be a linear algebraic group over $\kk$ and let $V$ be a $G$-module. Following~\cite{ElmerKohlsSigmaDeltaInfinite} we define for any $v \in V$
\begin{equation*}
\epsilon(G,v): = \inf\{d\in\N_{>0}\mid
  \text{ there exists } f \in \kk[V]_d^G \text{ such that } f(v) \neq 0\},
\end{equation*}
where the infimum of an empty set is infinity. Thus, $G$ is reductive if $\epsilon(G,v)$ is finite for all nonzero $v \in V^G$ and linearly reductive if $\epsilon(G,v) = 1$ for all nonzero $v \in V^G$.
Nagata and Miyata \cite[Proof of Theorem 1]{NagataMiyata}  also proved that if $v \in V^G$ and there exists $f \in \kk[V]^G_d$ such that $f(v) \neq 0$ with $d$ invertible in $\kk$, then there exists $\tilde{f} \in \kk[V]^G_1$ such that $f(v) \neq 0$. Consequently for any nonzero $v \in V^G$, $\epsilon(G,v)$ is equal to one, divisible by $p=\chr(\kk)$ or infinite.  In particular, if $\chr(\kk) = 0$ then every (geometrically) reductive group over $\kk$ is linearly reductive. A version of their argument rephrased in language consistent with this note can be found in \cite[Proposition 2.1]{ElmerKohlsSigmaDeltaInfinite}.  The main purpose of this article is to prove the following result generalising the above in the case of positive characteristic:

\begin{Theorem}\label{MainTheorem}
Let $p = \chr(\kk) \ge 0$, $r\ge 0$ an integer and $d\ge 1$ an integer invertible in $\kk$. Let $v \in V^{G}\setminus\{0\}$ be a nonzero fixed point and suppose there exists a homogeneous invariant $f$ of degree $p^{r}d$ such that $f(v) \neq 0$. Then there exists
a homogeneous invariant $\tilde{f}$ of
degree $p^{r}$ such that $\tilde{f}(v) \neq 0$. In particular, for any $v \in
V^{G}$ we have that $\epsilon(G,v)$ is either a power of $p$ (including
$p^{0}=1$) or $\infty$.
\end{Theorem}

One says that a pair of points $v,w \in V$ can be \emph{separated} if there exists an invariant $f\in\kvg$ such that $f(v)\ne f(w)$. It has become quite popular recently to investigate so called \emph{separating sets}, which are subsets $S$ of the invariant ring with the following property: whenever two points can be separated, then they
can be separated by an element of $S$. This research topic was introduced by Derksen
and Kemper \cite[Definition 2.3.8]{DerksenKemper}, and quite a number of papers have appeared which deal with this topic. Remarkably, it turns out that even if the ring of invariants $\kvg$ is not finitely generated, it still contains a finite separating set, see \cite[Theorem~2.3.15]{DerksenKemper}. 
From the point of view of this research topic, we deal with separating a fixed point $v\in
V^{G}$ from the zero point $w=0$.
Recall that, for any $G$ and $V$, \emph{Hilbert's Nullcone} $\NGV$ is defined to be the vanishing set of all homogeneous invariants of positive degree. It is natural to consider the quantity
\begin{equation*}
\delta(G,V): = \sup\left(\{\epsilon(G,v)\mid v \in V^G \setminus \NGV \}\cup\{0\}\right).
\end{equation*}
Since a separating set
must certainly contain an invariant separating a given point outside the
nullcone from zero, \cite[Theorem~2.3.15]{DerksenKemper} implies that
$\delta(G,V)$ is finite for any $G$ and $V$. If $G$ is linearly reductive
then $\delta(G,V) \le 1$ for all $V$. For this reason, the number
$\delta(G,V)$ can be considered as a measure for the \emph{``degree of reductivity''} of the representation $V$. Further results on $\delta(G,V)$ can be found in \cite{ElmerKohlsSigmaDelta}, \cite{ElmerKohlsSigmaDeltaInfinite} and \cite{KohlsSezerDegRed}.  Our main theorem implies immediately

\begin{Corollary}\label{CorMainTheorem}
For any $G$-module $V$, we have that $\delta(G,V)$ is zero, one, or a power of $p = \chr(\kk)$.
\end{Corollary}

This article is organised as follows: in section two we prove Theorem \ref{MainTheorem}. In section three we give an example showing how the theorem may be used to compute $\delta(G,V)$ in cases where the ring of invariants is difficult to compute.

\section{Separating fixed points from zero}

Before we prove our main result, we want to reproduce the argument showing
that in positive characteristic $p$, given a reductive group $G$ and a nonzero fixed point $v\in V^{G}$,
there exists an invariant of $p$-power degree separating $v$ from zero. This
is a consequence of the following standard result for reductive
groups.

\begin{Theorem}[{see \cite[Lemma A1.2]{Mumford}}]
Let $G$ be a reductive group over a field of positive characteristic $p$ and
$V,W$ be $G$-modules. If $\phi: \kk[V]\rightarrow\kk[W]$ is a surjective
$G$-equivariant algebra-homomorphism, then for any $f\in\kk[W]^{G}$ there exists an
$r\ge 0$ such that $f^{p^{r}}\in\phi(\kk[V]^{G})$.
\end{Theorem}

Now consider $v\in V^{G}\setminus\{0\}$. We define $W:=\kk v$ and write
$\kk[W]=\kk[x]$. The restriction
map $\phi: \kk[V]\rightarrow \kk[W]$, $f\mapsto f|_{W}$ is clearly surjective
and $G$-equivariant, and as $x\in\kk[W]^{G}$ the theorem implies the
existence of an invariant $f\in\kk[V]^{G}$ such that $f|_{W}=x^{p^{r}}$ for
some $r\ge 0$. It follows that for $h$ the degree $p^{r}$-component of $f$, we
 also have $h|_{W}=x^{p^{r}}$, and $h$ is a homogeneous invariant of degree
 $p^{r}$ satisfying $h(v)=h|_{W}(v)=x^{p^{r}}(v)=1\ne 0$. 

Note that although  this result implies that every nonzero fixed point $v$ can be separated from zero  by an invariant of $p$-power degree, it does not  imply that the
 minimal possible degree of an invariant separating $v$ from zero is also a
$p$-power. This is a consequence of Theorem \ref{MainTheorem} which we prove now. The proof is based on similar ideas to those used by Nagata and Miyata; indeed, it specialises to their proof in the case $r=0$.

\begin{proof}[Proof of Theorem \ref{MainTheorem}.]
We extend $v_{0}:=v$ to a basis $\{v_{0},v_{1},\ldots,v_{n}\}$ of $V$ and
form  the  corresponding dual basis $\{x_{0},x_{1},\ldots,x_{n}\}$ of
$V^{*}$. Then $f$ has the form $f=\sum_{i=0}^{p^{r}d}x_{0}^{p^{r}d-i}c_{i}$,
where $c_{i}\in \kk[x_{1},\ldots,x_{n}]_{i}$ and
$f(v_{0})=c_{0}\in\kk\setminus\{0\}$. Dividing by $c_{0}$, we may assume
$c_{0}=1$. We claim that
\[
\tilde{f}:=x_{0}^{p^{r}}+\frac{1}{d}\sum_{i=1}^{p^{r}}x_{0}^{p^{r}-i}c_{i}
\]
has the properties we require. Clearly, $\tilde{f}(v_{0})=1\ne 0$, and $\tilde{f}$ is
homogeneous of degree $p^{r}$. It remains to show
that $\tilde{f}$ is invariant. We will obtain $\tilde{f}$ as the image of $f$
under a $G$-equivariant map $\kk[V]_{p^{r}d}\rightarrow \kk[V]_{p^{r}}$. We have $V^{*}=\kk[x_{1},\ldots,x_{n}]_{1}\oplus\kk x_{0}$
as vector spaces. 
Since $v_{0}\in V^{G}$, we have that $\kk[x_{1},\ldots,x_{n}]_{1}$ is a $G$-submodule of
$V^{*}$. Consider the $G$-module $\kk[x_{0},\ldots,x_{n}]_{p^{r}d}$ and the subspace
\[
T:=\bigoplus_{i={p^{r}+1}}^{p^{r}d}\kk[x_{0}]_{p^{r}d-i}\otimes\kk[x_{1},\ldots,x_{n}]_{i},
\]
i.e. the set of polynomials of $\kk[x_{0},\ldots,x_{n}]_{p^{r}d}$ which have total
degree at least $p^{r}+1$ in the variables $x_{1},\ldots,x_{n}$. As $v_{0}\in
V^{G}$ we have, for any $g\in G$,
\[
g(x_{0})=x_{0}+\gamma(g) \quad\quad\text{ for some
}\quad\gamma(g)\in\kk[x_{1},\ldots,x_{n}]_{1}.
\]
It follows that $T$ is in fact a $G$-submodule of
$\kk[x_{0},\ldots,x_{n}]_{p^{r}d}$. We next show that the map
\[
\phi: \kk[x_{0},\ldots,x_{n}]_{p^{r}d}/T \mapsto
\kk[x_{0},\ldots,x_{n}]_{p^{r}}
\]
given by $\kk$-linear extension of
\[
\begin{array}{rcl}
x_{0}^{p^{r}d}+T&\mapsto& x_{0}^{p^{r}}\\
x_{0}^{p^{r}d-k}b_{k}+T&\mapsto&\frac{1}{d}x_{0}^{p^{r}-k}b_{k}\quad\text{ for
}b_{k}\in\kk[x_{1},\ldots,x_{n}]_{k}\,\,\text{ and }\,\,k=1,\ldots,p^{r}
\end{array}
\]
is an isomorphism of $G$-modules. Clearly $\phi$ is an isomorphism of
$\kk$-vector spaces, so it remains to show that $\phi$ is $G$-equivariant, i.e. $\phi(g(m+T))=g(\phi(m+T))$ for every $g\in G$ and
$m\in\kk[x_{0},\ldots,x_{n}]_{p^{r}d}$. By $\kk$-linearity, it is enough to
consider the cases $m=x_{0}^{p^{r}d}$ and $m=x_{0}^{p^{r}d-k}b_{k}$ for
$b_{k}\in\kk[x_{1},\ldots,x_{n}]_{k}$ and $k=1,\ldots,p^{r}$. Assume
$m=x_{0}^{p^{r}d}$ first. We fix $g\in G$, set $\gamma:=\gamma(g)$ and compute
\[
\phi(g(x_{0}^{p^{r}d}+T))=\phi((x_{0}+\gamma)^{p^{r}d}+T)=\phi((x_{0}^{p^{r}}+\gamma^{p^{r}})^{d}+T)\stackrel{(*)}{=}\]\[
\phi(x_{0}^{p^{r}d}+dx_{0}^{p^{r}(d-1)}\gamma^{p^{r}}+T)
=x_{0}^{p^{r}}+\gamma^{p^{r}}=(x_{0}+\gamma)^{p^{r}}=
\]\[g(x_{0}^{p^{r}})=g(\phi(x_{0}^{p^{r}d}+T)).
\]
Note that in (*) we have used that $x_{0}^{p^{r}(d-j)}\gamma^{p^{r}j}\in T$ for $j\ge 2$. 
Secondly assume $m=x_{0}^{p^{r}d-k}b_{k}$ with $1\le k\le p^{r}$ and
$b_{k}\in\kk[x_{1},\ldots,x_{n}]_{k}$. We write
$\widetilde{b_{k}}:=g(b_{k})\in\kk[x_{1},\ldots,x_{n}]_{k}$. In the following
computation we will use that ${p^{r}d-k\choose j}\equiv {p^{r}-k\choose j}\mod
p$ for $k=1,\ldots,p^{r}$ and $j=0,\ldots,p^{r}-k$, see
Lemma~\ref{charpbinomial} below. We obtain
\[
\phi(g(x_{0}^{p^{r}d-k}b_{k}+T))=\phi((x_{0}+\gamma)^{p^{r}d-k}\widetilde{b_{k}}+T)=\]\[\phi\left(\sum_{j=0}^{p^{r}d-k}{p^{r}d-k\choose
  j}x_{0}^{p^{r}d-k-j}\gamma^{j}\widetilde{b_{k}}+T\right).
\]
Note that $\gamma^{j}\widetilde{b_{k}}\in
\kk[x_{1},\ldots,x_{n}]_{j+k}$. In particular, for $j\ge p^{r}+1-k$, we have
$x_{0}^{p^{r}d-k-j}\gamma^{j}\widetilde{b_{k}}\in T$, so in the sum above only summands
for $j=0,\ldots,p^{r}-k$ have to be taken into account. Also note that $k\ge 1$, so each term of
$\gamma^{j}\widetilde{b_{k}}$ is of positive degree. Now by the definition
of $\phi$ we obtain
\[
\phi(g(x_{0}^{p^{r}d-k}b_{k}+T))=\phi\left(\sum_{j=0}^{p^{r}-k}{p^{r}d-k\choose
  j}x_{0}^{p^{r}d-k-j}\gamma^{j}\widetilde{b_{k}}+T\right)=\]\[\sum_{j=0}^{p^{r}-k}{p^{r}d-k\choose
  j}\frac{1}{d}x_{0}^{p^{r}-k-j}\gamma^{j}\widetilde{b_{k}}\stackrel{\textrm{ Lemma~\ref{charpbinomial}}}{=}\frac{1}{d}\sum_{j=0}^{p^{r}-k}{p^{r}-k\choose
  j}x_{0}^{p^{r}-k-j}\gamma^{j}\widetilde{b_{k}}=
\]
\[
\frac{1}{d}(x_{0}+\gamma)^{p^{r}-k}\widetilde{b_{k}}=g(\frac{1}{d}x_{0}^{p^{r}-k}b_{k})=g(\phi(x_{0}^{p^{r}d-k}b_{k}+T)).
\]
This shows that $\phi$ is indeed $G$-equivariant. Now let \[\pi:
\kk[x_{0},\ldots,x_{n}]_{p^{r}d}\rightarrow\kk[x_{0},\ldots,x_{n}]_{p^{r}d}/T\]
denote the canonical projection, which is $G$-equivariant as $T$ is a
$G$-submodule. Then $\phi\circ \pi:\kk[x_{0},\ldots,x_{n}]_{p^{r}d}\rightarrow
\kk[x_{0},\ldots,x_{n}]_{p^{r}}$ is a $G$-equivariant map, and hence it maps the
invariant $f$ to the invariant
\[
\phi(\pi(f))=\phi\left(\pi(\sum_{i=0}^{p^{r}d}x_{0}^{p^{r}d-i}c_{i})\right)=\phi\left(\sum_{i=0}^{p^{r}d}x_{0}^{p^{r}d-i}c_{i}+T\right).
\]
As for $i\ge p^{r}+1$ we have $x_{0}^{p^{r}d-i}c_{i}\in T$, only the summands where
$i=0,\ldots,p^{r}$ need to be considered, so we obtain
\[
\phi(\pi(f))=\phi\left(\sum_{i=0}^{p^{r}}x_{0}^{p^{r}d-i}c_{i}+T\right)=\phi(x_{0}^{p^{r}d}+T)+\sum_{i=1}^{p^{r}}\phi(x_{0}^{p^{r}d-i}c_{i}+T)=
\]
\[
x_{0}^{p^{r}}+\frac{1}{d}\sum_{i=1}^{p^{r}}x_{0}^{p^{r}-i}c_{i}=\tilde{f}.
\]
Hence, $\tilde{f}$ is $G$-invariant.
\end{proof}

We have used the following
characteristic $p$-relation on binomial coefficients:

\begin{Lemma}\label{charpbinomial}
Assume $p$ is a prime and $d\ge 1$. Then we have
\[
{p^{r}d-k\choose j}\equiv {p^{r}-k\choose j}\mod
p\quad\quad\text{ for } k=1,\ldots,p^{r}\,\text{  and } j=0,\ldots,p^{r}-k.
\]
\end{Lemma}

\begin{proof}
We first recall the well known Theorem of Lucas on binomial coefficients modulo a
prime (see \cite{FineLucas} for a short proof):  if $a, b$ are integers with $p$-adic expansions
$a=\sum_{i=0}^{\infty}a_{i}p^{i}$ and
$b=\sum_{i=0}^{\infty}b_{i}p^{i}$, then \[{a \choose b}\equiv \prod_{i=0}^{\infty}{a_i\choose b_i}
\mod p.
\]
Of course, here almost all summands are zero and almost all factors are equal
to $1$, as ${m\choose 0}=1$ for all $m\ge 0$.
We now consider the base-$p$-expansions $j=\sum_{i=0}^{\infty}j_{i}p^{i}$,
$p^{r}d-k=\sum_{i=0}^{\infty}a_{i}p^{i}$ and
$p^{r}-k=\sum_{i=0}^{\infty}b_{i}p^{i}$, where all $j_{i},a_{i},b_{i}$ are
zero for large enough $i$, and $0\le j_{i},a_{i},b_{i}<p$ for all $i$. As $k\ge
1$ and $j\le p^{r}-k$, we have that $j_{i}=0$ for $i\ge r$. As
$p^{r}d-k=p^{r}-k+(d-1)p^{r}$, it follows that $a_{i}=b_{i}$ for $0\le i<
r$. We thus have by Lucas' Theorem
\[
{p^{r}d-k\choose j}\equiv\prod_{i=0}^{\infty}{a_{i}\choose
  j_{i}}\equiv\prod_{i=0}^{r-1}{a_{i}\choose
  j_{i}}\equiv\prod_{i=0}^{r-1}{b_{i}\choose
  j_{i}}\equiv\prod_{i=0}^{\infty}{b_{i}\choose j_{i}}\equiv{p^{r}-k\choose
  j}\mod p.
\]
\end{proof}

\section{An example}

Corollary \ref{CorMainTheorem} sometimes allows a determination of $\delta(G,V)$
for a given representation $V$ without knowledge of the invariant ring. The special case $p=2$ of the following example was
treated in \cite[Proposition 12]{KohlsSezerDegRed}.

\begin{eg}
Consider a field of positive characteristic $p$, the cyclic group $Z_{p}$ of order $p$, and the action of the group $G=Z_{p}\times Z_{p}=\langle g_{1},g_{2}\rangle$ on a
$G$-module $V=\langle h_{1},\ldots,h_{m},e_{1},\ldots,e_{m}\rangle$, $m\ge 2$, where
$g_{1}$ acts by the matrix
$\left(\begin{array}{cc}I_{m}&0\\I_{m}&I_{m}\end{array}\right)$, and $g_{2}$
acts by the matrix
$\left(\begin{array}{cc}I_{m}&0\\J_{m}(\lambda)&I_{m}\end{array}\right)$. Here,
$I_{m}$ denotes the $m\times m$ identity matrix, and $J_{m}(\lambda)$ a lower
triangular $m\times m$ Jordan block with eigenvalue $\lambda\in \kk$. Then
$e_{m}\in V^{G}$, and we want to show that $\epsilon(G,e_{m})=p^{2}$. Note that, since for a finite group we have $\delta(G,V) \leq |G|$ (see \cite[Theorem~1.1]{ElmerKohlsSigmaDelta}), this shows that $\delta(G,V)= p^2$. 

We write
$\kk[V]=\kk[x_{1},\ldots,x_{m},y_{1},\ldots,y_{m}]$. We then have
\[
\begin{array}{rcll}
g_{i}(x_{j})&=&x_{j} &\text{ for }i=1,2,\,\,j=1,\ldots,m\\
g_{1}(y_{j})&=&y_{j}-x_{j}&\text{ for } j=1,\ldots,m\\
g_{2}(y_{1})&=&y_{1}-x_{1}\\
g_{2}(y_{j})&=&y_{j}-\lambda x_{j}- x_{j-1}&\text{ for } j=2,\ldots,m.
\end{array}
\]
It is easy to see that
$\kk[V]^{G}_{1}=\langle x_{1},\ldots,x_{m}\rangle$, which shows
$\epsilon(G,e_{m})>1$. As $\epsilon(G,e_{m})$ is a $p$-power by Theorem
\ref{MainTheorem}, and bounded above by $|G|=p^{2}$, it suffices to show that
$\epsilon(G,e_{m})\ne p$. To this end, we will demonstrate that $y_{m}^{p}$ does
not appear in any invariant polynomial. Define 
\[\Delta_{i,j}:\kk[V]\rightarrow
\kk[V],\quad  f\mapsto g_{1}^{i}g_{2}^{j}(f)-f\quad\quad\text{ for
}i,j\in\Z.\]
 Then for an invariant polynomial $f$, $\Delta_{i,j}(f)=0$ for all $i,j$. We will say that a
monomial $r$ \emph{lies over} a monomial $s$ with respect to $\Delta_{i,j}$ 
if $s$ appears in $\Delta_{i,j}(r)$. As $g_{1}^{i}g_{2}^{j}$ acts by the
matrix
$\left(\begin{array}{cc}I_{m}&-iI_{m}-{j}J_{m}(\lambda)^{T}\\0&I_{m}\end{array}\right)$
on $V^{*}$, it follows that  if a monomial $r$
lies over $x_{m}^{p}$ with respect to $\Delta_{i,j}$ for some $i,j$, then
$r$ is an element of the set $M:=\{y_{m}^{p},x_{m}y_{m}^{p-1},x_{m}^{2}y_{m}^{p-2},\ldots,x_{m}^{p-1}y_{m}\}$. Let
now $f\in\kk[V]^{G}$ be an invariant. Let
$h=\sum_{k=1}^{p}c_{k}y_{m}^{k}x_{m}^{p-k}$, $c_{k}\in\kk$, be the partial sum of terms of $f$
with monomials from $M$.
Then for all $i,j$, the coefficients of $x_{m}^{p}$
in $\Delta_{i,j}(f)=0$ and $\Delta_{i,j}(h)$ respectively are equal. From
\begin{eqnarray*}
\Delta_{-i,-1}(h)&=&\Delta_{-i,-1}\left(\sum_{k=1}^{p}c_{k}y_{m}^{k}x_{m}^{p-k}\right)\\
&=&\sum_{k=1}^{p}(c_{k}(y_{m}+(\lambda+i)x_{m}+x_{m-1})^{k}x_{m}^{p-k}-c_{k}y_{m}^{k}x_{m}^{p-k})\\&=&\ldots+\sum_{k=1}^{p}c_{k}(\lambda+i)^{k}x_{m}^{p}+\ldots
\end{eqnarray*}
it follows that $\sum_{k=1}^{p}c_{k}(\lambda+i)^{k}=0$ for
$i=0,\ldots,p-1$. Therefore all elements of the set
$Z:=\{\lambda,\lambda+1,\ldots,\lambda+p-1\}$ of size $p$ are roots of the polynomial $q:=\sum_{k=1}^{p}c_{k}X^{k}$.
Clearly, $0$ is also a root of $q$. Assume first $0\not\in Z$. Then the
polynomial $q$ of degree $\le p$ has the elements of $\{0\}\cup Z$ as $p+1$
different roots, i.e. $q=0$. In particular, $c_{p}=0$, which shows that
$y_{m}^{p}$ does not appear in $f$, which we wanted to prove and we are
done. Secondly assume  $0\in Z$. It follows $\lambda+i_{0}=0$ for some
$i_{0}\in\{0,\ldots,p-1\}$, which implies $Z=\{0,1,2,\ldots,p-1\}$. As $Z$ is also the
set of roots of $X^{p}-X$, it follows $q=c(X^{p}-X)$ for some $c\in\kk$,
i.e. $c_{p}=c$, $c_{1}=-c$, and the other $c_{i}$'s are zero. Therefore we have $h=c(y_{m}^{p}-x_{m}^{p-1}y_{m})$. As $i_{0}+\lambda=0$, $g_{1}^{-i_{0}}g_{2}^{-1}$ acts by the matrix $\left(\begin{array}{cc}I_{m}&i_{0}I_{m}+J_{m}(\lambda)^{T}\\0&I_{m}\end{array}\right)=\left(\begin{array}{cc}I_{m}&J_{m}(0)^{T}\\0&I_{m}\end{array}\right)$
on $V^{*}$. From this it can be seen that $x_{m}^{p-1}y_{m}$ is the only monomial that lies over
$x_{m}^{p-1}x_{m-1}$ with respect to $\Delta_{-i_{0},-1}$. Therefore, the coefficients
of $x_{m}^{p-1}x_{m-1}$ in $\Delta_{-i_{0},-1}(f)=0$ and
\[
\Delta_{-i_{0},-1}(-cx_{m}^{p-1}y_{m})=-cx_{m}^{p-1}(y_{m}+x_{m-1})+cx_{m}^{p-1}y_{m}=-cx_{m}^{p-1}x_{m-1}
\]
are equal, hence $0=c=c_{p}$, which shows that $y_{m}^{p}$ does not appear in $f$
 as claimed.
\end{eg}

\begin{rem}
In the above, it was easy to see that  $p^{2}\ge\epsilon(G,e_m) > 1$, and we
showed $\epsilon(G,e_m) \neq p$. Theorem \ref{MainTheorem} allowed us to
conclude that $\epsilon(G,e_m) =p^2$. If $p=2$ this follows straight away from
Nagata and Miyata's result, but if $p>2$ it is hard to rule out the
possibility that $\epsilon(G,e_m) = dp$ for some $1<d<p$ without using  our theorem.
\end{rem}

\begin{ack}
This paper was prepared during a visit of the first author to TU M\"unchen. We want to thank Gregor Kemper for making this visit possible.
\end{ack}

\bibliographystyle{plain}
\bibliography{MyBib}

\end{document}